\documentclass[12pt] {amsart}

\theoremstyle{plain}
\newtheorem{thm}{Theorem}[section]

\newtheorem{prop}[thm]{Proposition}

\theoremstyle{definition}

\newtheorem{rmk}[thm]{Remark}

\newtheorem{claim}[thm]{Claim}

\newcommand{\F}{{\mathcal F}}

\newcommand{\Q}{{\mathbb Q}}

\author{Christopher D. Hacon}
\address{Department of Mathematics, University of Utah, 155 South 1400 East,
Salt Lake City, UT 48112-0090, USA}
\email{hacon@math.utah.edu}

\begin{document}
\title{Singularities of pluri-theta divisors in Char $p>0$}
\begin{abstract}
  We show that if $(X,\Theta )$ is a PPAV over an algebraically closed
 field of characteristic $p>0$ and $D\in |m\Theta|$, then $(X,\frac 1 m D )$ is a limit of strongly $F$-regular pairs and in particular ${\rm mult}_x(D)\leq m\cdot \dim X$ for any $x\in X$.
\end{abstract}

\maketitle
\section{introduction} Let $(X,\Theta )$ be a principally polarized abelian variety (PPAV) so that $X$ is a connected projective algebraic group and $\Theta$ is an ample divisor with $h^0(X,\mathcal O _X (\Theta ))=1$. 
The geometry of $X$ is often studied in terms of the singularities of the theta divisor $\Theta$ (or more genereally of the singularities of pluri-theta divisors i.e. divisors in $|m\Theta|$).
For PPAVs over $\mathbb C$ there are a number of well known results saying that the singularities of pluri-theta divisors are mild (see for example \cite{Kollar95}, \cite{EL97} and \cite{Hacon99}).
According to a result of Ein and Lazarsfeld (cf. \cite[3.5]{EL97}), if $D\in |m\Theta|$, then $(X,\frac 1 m D )$ is log canonical. Since $X$ is smooth, this is equivalent to saying that $(X,\frac {1-\epsilon} m D )$ is Kawamata log terminal for any $0<\epsilon <1$.
The purpose of this brief note it to prove the analogous result in characteristic $p>0$.
\begin{thm}\label{t-main} Let $(X,\Theta )$ be a PPAV 
over an algebraically closed field of characteristic $p>0$. If $D\in |m\Theta|$, then $(X,\frac {1-\epsilon} m D)$ is strongly $F$-regular for any rational number $0<\epsilon <1$.
\end{thm}
\begin{rmk}In particular it follows that $(X,\frac {1}mD)$ is log canonical in the sense that all log-discrepancies are $\geq 0$ and therefore ${\rm mult }_x(D)\leq m\cdot \dim X$ for any $x\in X$.
Note that in characteristic $0$ it is known that if ${\rm mult} _x(D)= m\cdot \dim X$, then $X$ is a product of elliptic curves \cite{EL97} and \cite{Hacon99}. We do not know if the analogous result holds in characteristic $p>0$.
By \cite[4.1]{Hernandez11} it follows that  $(X,D)$ is $F$-pure, and if $p$ and $m$ are coprime, then $(X,D)$ is sharply $F$-pure.
\end{rmk}
\begin{rmk} Since in characteristic $0$ it is known that irreducible $\Theta$ divisors are normal with rational
singularities, it is natural to wonder if over an algebraically closed field of  characteristic $p>0$,    irreducible $\Theta$ divisors are normal with $F$-rational
singularities. Note that a related result appears in \cite{BBE07}.\end{rmk}
The characteristic $0$ argument of Ein and Lazarsfeld relies on the theory of multiplier ideal sheaves, Kawamata-Viehweg vanishing and the Fourier-Mukai functor. In characteristic $p>0$, the theory of Fourier-Mukai functors still applies and multiplier ideal sheaves can be replaced by test ideals.
However, there is no good substitute for Kawamata-Viehweg vanishing (which is known to fail in this context). Instead, inspired by some ideas contained in \cite{Schwede11}, we use the  "generic vanishing results" from \cite{CH03}, \cite{Hacon04} and \cite{PP08} to show that $h^0(X,\sigma (X,\frac {1-\epsilon}mD)\otimes \mathcal O _X(\Theta )\otimes P_{\hat x})>0$ for a general $\hat x\in \widehat X$.
But then a general translate of $\Theta$ vanishes along the cosupport of $\sigma (X,\frac {1-\epsilon}mD)$.
This is only possible if $\sigma (X,\frac {1-\epsilon}mD)=\mathcal O _X$ and so $(X,\frac {1-\epsilon}mD)$ is strongly $F$-regular and \eqref{t-main} follows.

\hskip.3cm

\noindent{\bf Acknowledgments.}
  The author was partially supported by NSF research grant DMS
  0757897. He would  also like to thank Mihnea Popa, Karl Schwede and Kevin Tucker
   for some useful discussions.
\section{preliminaries}
Throughout this paper we work over an algebraically closed field $k$ of characteristic $p>0$.
Recall that the ring homomorphism $F:k\to k$ defined by $F(x)=x^p$ endows $k$ with a non-trivial $k$-module structure. 
\subsection{Test ideals}\label{ss-t} Here we recall the definition of test ideals and some related results that will be needed in this paper. We refer the reader to \cite{BST11} and \cite{Schwede11} (and the references therein) for a more complete treatment.
Let $(X, \Delta =\sum d_iD_i)$ be a log pair so that $X$ is a normal variety and $\Delta \geq 0$ is a $\Q$-divisor such that $K_X+\Delta$ is $\Q$-Cartier.
Let $F:X\to X$ be the Frobenius morphism and for any integer $e>0$, let $F^e$ be its $e$-th iterate.
The {\bf parameter test submodule} of $(X,\Delta )$ denoted by $\tau(\omega _X,\Delta )$ is locally defined as the unique smallest non-zero $\mathcal O _X$-submodule of $\omega _X$ such that $\phi (F^e_*M)\subset M$ for any $e>0$ and any $\phi \in {\rm Hom}_{\mathcal O _X}(F^e_*\omega _X(\lceil (p^e-1)\Delta \rceil),\omega _X)$. The {\bf test ideal} $\tau (X,\Delta )$ is defined by $\tau (\omega _X, K_X+\Delta )$.
It is known that $\tau (X,\Delta )\subset \mathcal O _X$ is an ideal sheaf such that $$\tau  (X,\Delta +A)=\tau (X,\Delta )\otimes \mathcal O_X(-A),\ \ \  {\rm and} \ \ \ \tau  (X,\Delta +eA)=\tau (X,\Delta )$$ for any Cartier divisor $A$ and any rational number $0\leq e\ll 1$. We also have that test ideals are contained in multiplier ideals in the sense that if $\pi :Y\to X$ is a proper birational morphism, then $\tau (X,\Delta )\subset \pi_* \mathcal O_Y(K_{Y}-\lfloor \pi ^*(K_X+\Delta )\rfloor )$. (Recall that in characteristic $0$, if $\pi$ is a log resolution, then the multiplier ideal is defined by $\mathcal J(X,\Delta )=\pi_* \mathcal O_Y(K_{Y}-\lfloor \pi ^*(K_X+\Delta )\rfloor )$.)
In particular,
if $X$ is a smooth variety and ${\rm mult} _x (\Delta )\geq \dim X$, then $\tau (X,\Delta )\subset \frak m _x$ where $\frak m _x$ is the maximal ideal of $x$ in $X$.

Suppose now that $p$ does not divide the index of $K_X+\Delta$ so that $(p^e-1)(K_X+\Delta)$ is Cartier for some integer $e>0$. Let $\mathcal L_{e,\Delta }=\mathcal O_X((1-p^e)(K_X+\Delta ))$. There is a canonically determined (up to unit) homomorphism of line bundles $\phi _\Delta :F^e_*\mathcal L_{e,\Delta }\to \mathcal O_X$. We have that $\tau (X,\Delta )$ is  the smallest non-zero ideal $J\subset \mathcal O _X$ such that 
$\phi _\Delta (F^e_*(J\cdot \mathcal L_{e,\Delta }))=J$. Similarly, we define $\sigma (X,\Delta )$ to be the largest ideal $J\subset \mathcal O _X$ such that $\phi _\Delta (F^e_*(J\cdot \mathcal L_{e,\Delta }))=J$.
By definition $(X,\Delta )$ is {\bf strongly $F$-regular} if $\tau (X,\Delta )=\mathcal O _X$, and {\bf sharply $F$-pure} if $\sigma (X,\Delta )=\mathcal O _X$.
By \cite[2.2]{TW04}, we have that  1) if $(X,\Delta )$ is strongly $F$-regular then it is also sharply $F$-pure, and 2) if $(X,\Delta )$ is sharply $F$-pure and $X$ is strongly $F$-regular, then  $(X,(1-\epsilon)\Delta )$ is strongly $F$-regular for any $0<\epsilon <1$.
\subsection{Abelian varieties and the Fourier-Mukai transform} Here we recall some facts about the Fourier-Mukai transform introduced in \cite{Mukai81}.
Let $\widehat X$ be the dual abelian variety and $P$ be the normalized Poincar\'e line bundle on $X\times \widehat X$. We denote by $\mathbf R \hat S: \mathbf D(X)\to \mathbf D(\widehat X)$ the usual Fourier-Mukai functor given by $\mathbf R \hat S(\F )=\mathbf R {p_{\widehat X}}_*(p_X^*\F \otimes P)$. There is a corresponding functor  $\mathbf R  S: \mathbf D(\widehat X)\to \mathbf D( X)$ such that $$\mathbf R  S\circ \mathbf R \hat S=(-1_X)^*[-g]\qquad {\rm and }\qquad \mathbf R  \hat S\circ \mathbf R S=(-1_{\widehat X})^*[-g].$$ 

Let $A$ be any ample line bundle on $\widehat X$, then $\mathbf R ^0S(A)=\mathbf R S(A)$ is a vector bundle on $X$ of rank $h^0(A)$ which we denote by $\widehat A$.
For any $x\in X$, let $t_x:X\to X$ be the translation by $x$ and let $\phi _A:\widehat X \to X$ be the isogeny determined by $\phi _A (\hat x)=t^*_{\hat x}A-A$, then $\phi _A^*(\widehat{A})=\bigoplus _{h^0(A)}A^\vee$. 

If $(X,\Theta )$ is a PPAV, then $\phi _\Theta: (X, \Theta)\to (\widehat X,\widehat \Theta =\phi _\Theta (\Theta))$ is an isomorphism of PPAVs. If $A=\mathcal O _X(m\widehat \Theta )$ for some positive integer $m>0$, then $\phi _A:\widehat X \to X$ can be identified with $m_X:X\to X $ (multiplication by $m$) and so it has degree $m^{2\dim X}$ and $\phi _A^*\Theta \equiv m^2\widehat \Theta$. (Note that the above notation is customary, but somewhat confusing as $\mathcal O _{\widehat X}(-\widehat \Theta) ={\mathbf R}^0\hat S (\mathcal O _X (\Theta ))$.)

We will need the following result.
\begin{prop}\label{p1}Let $\F$ be a non-zero coherent sheaf on $X$ such that $H^i(\F \otimes P_{\hat x})=0$ for all $i>0$ and all $\hat {x}\in \widehat X$ (where for any $\hat x\in \hat X$, we let $P_{\hat x}=P|_{X\times \hat x}$). If $\F\to k(x)$ is a surjective morphism for some $x\in X$, then the induced map $H^0(\F \otimes P_{\hat x})\to H^0(k(x)\otimes P_{\hat x})\cong k(x)$ is surjective for general $\hat {x}\in \widehat X$.
\end{prop}
\begin{proof} (Cf. \cite[2.3]{CH03} or \cite{PP08}.)
By cohomology and base change, one sees that $\hat \F=\mathbf R ^0S(\F )=\mathbf R S(\F )$ is a sheaf and since $\mathbf R \hat S (\hat \F)=(-1_X)^*\F[-g]\ne 0$, we have that $\hat \F\ne 0$.
Let $P_x=P|_{x\times \hat X}$, then $P_x=\mathbf R ^0S(k(x))=\mathbf R S(k(x))$ and the homomorphism
$\phi:\hat \F \to P_x$ is non-zero. However, as $P_x$ is a line bundle (and hence torsion free of rank $1$), it follows that $\phi$ is generically surjective. The proposition now follows since for any $\hat x\in \hat X$, the corresponding fiber of $\hat \F$ (resp. $\mathbf R S^0(k(x))$) is $H^0(\F\otimes P_{\hat x})$ (resp. $H^0(k(x)\otimes P_{\hat x})=H^0(k(x))$).
\end{proof}
\section{main result}

\begin{proof} [Proof of \eqref{t-main}]

Let $0<\epsilon <1$ be a rational number such that the index of $\Delta =\frac {1-\epsilon}mD$  is not divisible by $p$. We will show that $(X,\Delta )$ is sharply $F$-pure. Note that for any given $m$, we may find a sequence of $\epsilon _i$ such that the index of $\Delta =\frac {1-\epsilon_i}mD$  is not divisible by $p$ and $0=\lim _{i\to \infty}\epsilon _i$.  
Since $X$ is regular, it follows by what we have observed in Subsection \ref{ss-t}, that $(X,(1-\epsilon )\Delta)$ is strongly $F$-regular for all rational numbers $0<\epsilon \leq 1$.

We now fix $e>0$ such that $(p^e-1)\Delta $ is Cartier. For any $n> 0$ we have that $$\phi _\Delta ^{n}\left(F^{ne}_*(\sigma (X,\Delta)\cdot \mathcal O _X((1-p^{ne})(K_X+\Delta)))\right)=\sigma (X,\Delta).$$
We will show the following.
\begin{claim}\label{c-s}  For any sufficiently big integer $n\gg 0$,  we have $$H^i(X,F_*^{ne} (\sigma (X,\Delta )\cdot \mathcal O_X((1-p^{ne})(K_X+\Delta)))\otimes \mathcal O _X (\Theta)\otimes P_{\hat {x}})=0$$ for all $i>0$ and all $\hat x\in \widehat X$.
\end{claim}
Granting the claim for the time being, we will now conclude the proof of the theorem.
Let $x\in X$ be a general point, so that in particular $x$ is not contained in the co-support of $\sigma (X,\Delta )$. We have a surjection $$F^{ne}_*(\sigma (X,\Delta)\cdot \mathcal O _X((1-p^{ne})(K_X+\Delta)))\otimes \mathcal O _X(\Theta)\to k(x),$$
which factors through $\sigma (X,\Delta)\otimes \mathcal O _X(\Theta)\to k(x).$
By \eqref{p1} and \eqref{c-s}, we have that $$H^0(F^{ne}_*(\sigma (X,\Delta)\cdot \mathcal O _X((1-p^{ne})(K_X+\Delta)))\otimes \mathcal O _X (\Theta)\otimes P_{\hat x})\to k(x)$$ is surjective for general $\hat x\in \widehat X$.
Since this map factors through $H^0(\sigma (X,\Delta)\otimes \mathcal O _X(\Theta)\otimes P_{\hat {x}})$, we have that the induced homomorphism $H^0(\sigma (X,\Delta)\otimes \mathcal O _X(\Theta)\otimes P_{\hat {x}})\to k(x)$ is surjective. In particular $H^0(\sigma (X,\Delta)\otimes \mathcal O _X(\Theta)\otimes P_{\hat {x}})\ne 0$, 
i.e. the corresponding translate of $\Theta$ vanishes along the co-support of $\sigma (X,\Delta)$.
But then this co-support is empty so that $\sigma (X,\Delta)=\mathcal O _X$. 
\end{proof}
\begin{proof}[Proof of Claim \ref{c-s}.] 
It suffices to show that for all $i>0$ we have $$H^i( {X}, \sigma (X,\Delta )\cdot \mathcal O_X((1-p^{ne})(K_X+\Delta))\otimes {F^{ne}}^*(\mathcal O _X (\Theta)\otimes P_{\hat x}))=0.$$
By Serre vanishing (applied to the projective morphism $p_{\widehat X}: X\times \widehat X\to \widehat X$ and the coherent sheaf $p _{\widehat X}^*(\sigma (X,\Delta )\cdot  (F^{ne}\times {\rm id}_{\hat X})^*P$), we may fix $t>0$ such that $$H^i( X, \sigma (X,\Delta )\cdot  {F^{ne}}^* P_{\hat x}\otimes \mathcal O _X (t\Theta))=0$$ for $i>0$.
By \cite[2.9]{PP03}, it suffices to show that $$H^i( {X}, \mathcal O_{X}((1-p^{ne})(K_X+\Delta)+(p^{ne}-t)\Theta) )=0$$
for $i>0$. 
By assumption $(1-p^{ne})(K_X+\Delta)+(p^{ne}-t)\Theta \sim _{\mathbb Q}((p^{ne}-1)\epsilon +1-t)\Theta$. 
The claim now follows since $(p^{ne}-1)\epsilon +1-t >0$ for $ n\gg 0$.
\end{proof}


\begin{thebibliography}{ELMNPM}

\bibitem[BBE07]{BBE07}
P. Berthelot and S. Bloch and H. Esnault, {\it On Witt vector cohomology for singular varieties}, Compositio math. 143 (2007), 363-392. 
\bibitem[BST11]{BST11}
M. Blickle, K. Schwede, K. Tucker, {\it $F$-singularities via alterations.} ArXiv:1107.3807v2
\bibitem[CH03]{CH03}
J. A. Chen and C. D. Hacon, {\it Linear series on irregular varieties.}  Proceedings of the symposium on Algebraic Geometry in East Asia 143--153. World Scientific (2002).
\bibitem[EL97]{EL97}
L. Ein and R. Lazarsfeld, {\it Singularities of theta divisors, and the birational geometry of irregular varieties.}
J. Amer. Math. Soc. {\bf 10} (1997), 243-258.
\bibitem[Hacon99]{Hacon99}
C. Hacon, {\it Divisors on principally polarized abelian varieties}. Compositio Mathematica {\bf 119}: 321--329, 1999.
\bibitem[Hacon04]{Hacon04}
C. Hacon, {\it A derived category approach to generic vanishing.}  J. reine angew. Math. {\bf 575} (2004), 173--187.
\bibitem[Kollar95]{Kollar95} J. Koll\'ar, {\it Shafarevich maps and automorphisc forms.} Princeton Univ. Press, 1995.
\bibitem[Hernandez11]{Hernandez11}
D. Hernandez, {\it F-purity of hypersurfaces.} Preprint. (2011)
\bibitem[Mukai81]{Mukai81}
S. Mukai, {\it Duality between $D(X)$ and $D(\hat X)$ with its application to Picard sheaves.} Nagoya Math. J. Vol. 81 (1981), 153--175.
\bibitem[PP03]{PP03}
G. Pareschi and M. Popa, {\it Regularity on abelian varieties I} J. Amer. Math. Soc. 16 (2003), no.2, 285-302 
\bibitem[PP08]{PP08}
G. Pareschi and M. Popa, {\it Regularity on abelian varieties III: relationship with Generic Vanishing and applications.} arXiv:0802.1021
\bibitem[Schwede11]{Schwede11}
K. Schwede, {\it A canonical linear system associated to adjoint divisors in characteristic $p>0$.} arXiv:1107.3833v2
\bibitem[TW04]{TW04}
S. Takagi and K-I Watanabe, {\it On F-pure thresholds.} Journl of Algebra 282 (2004) 278-297
\end{thebibliography}
\enddocument
\end